\documentclass[a4paper, 12pt]{article}
\usepackage{amssymb}
\usepackage[english]{babel}
\usepackage{epsfig}
\usepackage{graphicx}
\usepackage{amsmath}
\usepackage{pdfsync}

\large\normalsize \hbadness3000 \vbadness30000
\def\mytitle#1{\setcounter{equation}{0}
\setcounter{footnote}{0}
\begin{center}\Large\textbf{#1}\end{center}
\vspace{0.25cm}}
\def\myname#1{\centerline{{\large #1}}\vspace{-0.13cm}}

\newtheorem{theorem}{Theorem}[section]

\newtheorem{conjecture}[theorem]{Conjecture}
\newtheorem{corollary}[theorem]{Corollary}

\newenvironment{proof}[1][Proof]{\noindent\textbf{#1: }}{\hspace{\stretch{1}}\rule{0.5em}{0.5em}}

\begin{document}
\mytitle{Bounds on  Erd{\H{o}}s - Faber - Lov\'{a}sz Conjecture - the Uniform and Regular Cases}
\myname{Suresh Dara$^{a,b}$, S. M. Hegde$^c$}

%
%
%

\begin{abstract}
We consider the  Erd{\H{o}}s - Faber - Lov\'{a}sz (EFL) conjecture for hypergraphs. This paper gives an upper bound for the chromatic number of $r$ regular linear hypergraphs $\textbf{H}$ of size $n$. If $r \ge 4$, $\chi (\textbf{H}) \le 1.181n$ and if $r=3$, $\chi(\textbf{H}) \le 1.281n$.
\end{abstract}



\section{Introduction}

A \textit{hypergraph} is a structure $\textbf{H}= (V, (E_i:i \in I))$ where the vertex set $V$ is an arbitrary set, and every $E_i \subseteq V$. These sets $E_i$ are called the hyperedges of the hypergraph.

The \textit{degree} of a vertex $v$ in $\textbf{H}$ is the number of edges $d(v)$ containing $v$. The rank of an edge $E$ is the cardinality $r(e)$ of $e$.
%
A hypergraph is said to be \textit{linear} if no two hyperedges have more than one vertex in common. A hypergraph is said to be \textit{uniform} if all of its hyperedges have the same rank. If the degree of each vertex is same then the hypergraph is called \textit{regular}.  

The \textit{dual} hypergraph $\textbf{H}^*$ of a hypergraph $\textbf{H}$ is the transpose of an incident matrix of the hypergraph $\textbf{H}$. Clearly the edges of $\textbf{H}^*$ are the vertices of $\textbf{H}$ and vice versa, ranks swap with degrees, etc. The dual of a uniform hypergraph is regular and vice versa. It is easy to see that H is linear if and only if $\textbf{H}^*$ is linear.

A \textit{coloring of a hypergraph} is an assignment of colors to the vertices so that no two vertices of an edge has the same color. A \textit{$k$-coloring} of a hypergraph is a coloring of it where the number of used colors is at most $k$.

In 1972, Erd{\H{o}}s - Faber - Lov\'{a}sz (EFL) conjectured as follows:

\begin{conjecture}\label{EFL1}
If $\textbf{H}$ is a linear hypergraph consisting of $n$ edges of cardinality $n$, then it is possible to color the vertices with $n$ colors so that no two vertices with the same color are in the same edge\cite{berge1990onvizing}.
\end{conjecture}

\begin{conjecture}\label{EFL2}
	Let $\textbf{H}$ be a linear hypergraph with $n$ vertices and no rank 1 edges. Then
	$q(H) \le n$ \cite{faber2016linear} .
\end{conjecture}

%
%

Chang and Lawler \cite{chang1988edge} presented a simple proof that the edges of a simple hypergraph on $n$ vertices can be colored with at most [1.5n-2] colors. Kahn \cite{kahn1992coloring} showed that the chromatic number of $\textbf{H}$ is at most $n+o(n)$. Faber \cite{faber2010uniformregular} proves that for fixed degree, there can be only finitely many counterexamples to EFL on this class of regular and uniform hypergraphs.

In this paper we are using the dual graph version of the Conjecture \ref{EFL2}, we gave an upper bound for the chromatic number of $r$ regular linear hypergraphs $\textbf{H}$ of size $n$. If $r \ge 4$, $\chi (\textbf{H}) \le 1.181n$ and if $r=3$, $\chi(\textbf{H}) \le 1.281n$.

\section{Results}

\begin{theorem}
	Let $H$ be a linear hypergraph of size $n$. 
	\begin{enumerate}
		\item If $H$ is $r$ regular $(r\ge 4)$ then $\chi(H) \leq 1.181n$.
		\item If $H$ is $3$ regular then $\chi(H) \leq 1.281n$.
	\end{enumerate}
	\end{theorem}

\begin{proof}
	Let $H= (V,E)$ be a linear hypergraph of size $n$ and $H$ is $r$ regular $(r\ge 3)$. Let $E_1, E_2, \dots E_n$ be the edges of $H$. Since $H$ is $r$ regular and $|E| = n$, for every $E_i \in E$, $|E_i| \leq  \left \lfloor {\frac{n-1}{r-1}} \right \rfloor \leq \frac{n-1}{2}$, then $|V| \leq \frac{n(n-1)}{r(r-1)} \leq  \frac{n}{3}\left \lfloor {\frac{n-1}{r-1}} \right \rfloor = N (say)$. Let $C=\{1, 2, 3, \dots , \alpha n\}$ be the set of available colors. Since degree of each vertex is $r$, for every vertex $v_i \in V$, there exists $r$ edges  $E_{i_1}, E_{i_2}, \dots  E_{i_r}$ which are incident at the vertex $v_i$.
	
	For $i=1$ to $N$ do
	
	Color $v_i$ with smallest color not already seen in $\cup_j E_{i_j}$.
	
	If no such color, then find a color $c$ such that
	
	$\forall$ $v_k \in \cup_j E_{i_j}$ with $col(v_k) = c$ (at most $r$ such $v_k$)
	
	$\exists$ colors $c_k \notin$ colors assigned to $\cup_{j=1,2,\dots,r}E_{k_j}$.
	
	Recolor each such $v_k$ with $c_k$ and color $v_i$ with $c$.
	
	If no such $c$ found, abort.
	
	Claim: For $\alpha$ sufficiently large, $\alpha \geq 1$, procedure dose not abort.
	
	Suppose procedure aborts at $v_i$. That means all $\alpha n$ colors seen in $\cup_j E_{i_j} \setminus \{v_i\}$ at vertices that cannot be recolored. Pick one such vertex for each color, set $S$. $|S| = \alpha n$. For vertices $v\in S$, $v$ sees all $\alpha n$ colors. At most $|E_{i_j}|$ of these are seen in the edge $E_{I_j} \ni v$. The rest are seen outside $E_{i_j}$. $v$ places a token on each vertex it sees colored with a color not in $E_{i_j}$.
	
	$v$ places $\geq \alpha n - |E_{i_j}|$ tokens.
	
	So total number of tokens places $\geq |S|.(\alpha n - |E_{i_j}|)$ $ \geq \alpha n (\alpha n -  \left \lfloor {\frac{n-1}{r-1}} \right \rfloor)$.
	
	How many tokens can be placed on a single vertex $u$?
	
	If $u \in \cup_j E_{I_j}$, then say $u \in E_{i_1}$. By $H$ is $r$ regular, $u$ belongs to $r-1$ other edges $E'_1, E'_2, \dots, E'_{r-1}$. By linearity property, at most one vertex each on $E_{i_2}, E_{i_3}, \dots, E_{i_{r}}$ lies on $E'_1$, and same for $E'_2$and so on. So, at most $(r-1)^2$ vertices of $S$ can place a token on $u$.
	
	If $u \notin  \cup_j E_{I_j}$, then $col(u)$ is seen in $S$, say on $ S \cap E_{i_1}$. then no vertex in $S \cap E_{i_1}$ will place a token on $u$. The $r$ edges through $u$ intersects $E_{i_2}, E_{i_3}, \dots, E_{i_r}$ in at most one vertex. So, there are at most $r(r-1)$ vertices of $S$ can place a token on $u$. Therefore number of token per vertex is at most $r(r-1)$.
	
	$|S|$. (minimum number of token placed by $v \in S$) $\leq$ (number of token) $\leq $ (number of vertices) (number of tokens per vertex) 
	
	$|S| (\alpha n - |E_{I_j}|) \leq $ (number of tokens) $\leq \frac{n}{r}  \left \lfloor {\frac{n-1}{r-1}} \right \rfloor r(r-1)$
	
	$\alpha n (\alpha n -  \left \lfloor {\frac{n-1}{r-1}} \right \rfloor) \leq n  \left \lfloor {\frac{n-1}{r-1}} \right \rfloor (r-1)$
	
	$\alpha^2 n - \alpha  \left \lfloor {\frac{n-1}{r-1}} \right \rfloor - (r-1)  \left \lfloor {\frac{n-1}{r-1}} \right \rfloor \leq 0$.
	
	Let $A= \left \lfloor {\frac{n-1}{r-1}} \right \rfloor$, then the above inequality is $ \alpha^2 n - \alpha A - (r-1)A \leq 0$.
	
	Consider $\alpha^2 n - \alpha A - (r-1)A = P(\alpha)$, then the zeros are $\frac{A \pm \sqrt{A^2 + 4(r-1)An}}{2n}$. $P(0) = -(r-1)A <0$, so for all $\alpha > \beta = \frac{A + \sqrt{A^2 + 4(r-1)An}}{2n}$, $P(\alpha) >0$.
	
	Choosing $\alpha$ so that $\alpha > \beta$, we conclude that with $\alpha n$ colors the procedure successfully colors $V$.
	
	The value of $\beta$:
	
	$\beta = \frac{A + \sqrt{A^2 + 4(r-1)An}}{2n}$ where $A = \left \lfloor {\frac{n-1}{r-1}} \right \rfloor \leq \frac{n}{r-1}= nB$, where $B=\frac{1}{r-1}$.
	
	$\leq  \frac{nB + \sqrt{n^2B^2 + 4n^2}}{2n}$ 
	
	$= \frac{B + \sqrt{B^2+4}}{2}$
	
	{\bf{Case 1:}} If $r\ge4$
	
	$\beta  \le \frac{B + \sqrt{B^2+4}}{2} \le 1.180460\dots < 1.181$ \quad ($B=\frac{1}{r-1}$)

Then for every $\alpha > \beta $, $\chi(H)\le \alpha n$. Therefore $\chi(H) \le 1.181n$.

	{\bf{Case 2:}} If $r = 3$

$\beta  \le \frac{B + \sqrt{B^2+4}}{2} $

$= \frac{1}{4} + \sqrt{\frac{1}{16}+1}$ \quad ($B=\frac{1}{r-1} = \frac{1}{2}$)

$=\sqrt{\frac{17}{16}} + 0.25$

$=\sqrt{1.0625} + 0.25$

$= 1.0307764 \cdots + 0.25 <1.281$.

Therefore $\chi(H) \le 1.281n$.

\end{proof}

From the above theorem one can observe that the upper bound of chromatic number of $r$ regular linear hypergraphs $\textbf{H}$ of size $n$ is depends on the value of $\beta$, where 	$\beta  \le \frac{B + \sqrt{B^2+4}}{2}$ \quad ($B=\frac{1}{r-1}$). Clearly as $r$ increases the value of $\beta$ will decrease. Therefore if $r\ge 5$, then the upper bound of $\chi (\textbf{H})$ is even smaller than $1.181n$.

\begin{theorem}
	Let $H$ be a linear uniform hypergraph of size $n$. If $\Delta (H) \geq \frac{n}{2}$ then $\chi(H) \leq 1.25n$.
\end{theorem}

\begin{proof}
	Let $H = (V,\ \{E_1, E_2, \dots E_n\})$ be a linear uniform hypergraph of size $n$ and  $\Delta (H) \geq \frac{n}{2}$. Let $v$ be a maximum degree vertex, $\delta(v)=m$ and $v$ is incident with the edges $E_1,E_2, \dots E_m$.  In any $E_i$ there are at most $\left \lfloor \frac{n-1}{d-1}  \right \rfloor$ vertices of degree $d$. Arrange the vertices of $H$ in non increasing order of degree. We will color the vertices in this order, using $[1.25n]$ colors. Assume we next color a vertex $v$ of degree $k \ge 4$. At this point only vertices of degree $k$ or greater have been assigned colors. At this stage in each edge incident with vertex $v$  there are at most $\left \lfloor \frac{n-1}{k-1} -1 \right \rfloor$ vertices have been colored, which implies there are at most $k\left \lfloor \frac{n-1}{k-1} -1 \right \rfloor < [1.25n]$ colors are used to the vertices incident to the edges which are incident to $v$. That means there will be an unused color for $v$. For vertices of degree 4, 3 and 2 we apply the following method.

	Partition the vertices of degree $d$ into two sets $A_d$ and $B_d$, where $A_d$ be the set of all degree $d$ vertices which are not incident with any of the edges $E_1, E_2, \dots, E_m$ and $B_d$ be the remaining degree $d$ vertices. First assign the colors to the vertics are in $A_d$, at this stage in any edge there are at most $\left \lfloor \frac{{n/2}}{d-1} -1 \right \rfloor$ vertices have been colored for $d=4,\ 3,\ 2$. Therefore always there is a free color for assigning to the vertices of $A_d$. For $d= 4, \ 3$, while assigning colors to the vertix $u$ from $B_d$, there are $(d-1)\left \lfloor \frac{n-1}{d-1} -1 \right \rfloor + \left \lfloor \frac{n/2}{d-1} -1 \right \rfloor < [1.25n]$ vertices have been colored. Therefore always there is a free color for assigning to the vertices of $B_d$. For $d=2$, let $u \in B_2$ be the vertex we have to color. Let $u \in E_i, E_j$ for some $i,j$, then the vertex $v$ is in either $E_i$ or $E_j$ but not both. Assume $v$ is in $E_i$, then in $E_i$ there are at most $\frac{n}{2} -1$ vertices have been colored, that means there are at least $\frac{3n}{4}+1$ colors are free from $E_i$. $E_j$ has at most $n-2$ vertices have been colored and it has at least $\frac{n}{4} +2$ colors free from $E_j$. Let $X$ be the set of free colors from $E_i$ and $Y$ be the set of free colors from $E_j$. If $E_i \cap E_j \neq \emptyset$ there is a free color to assign to the vertex $u$. If not, there exist two degree two vertices $p,q$ such that $p \in E_i$, $q\in E_j$ and $p,q \in E_k$ for some $k$ and color of $p$ in $Y$, color of $q$ in $X$. Sincethe  sum of number of free colors from $E_i$ and number of free colors from $E_j$ is $>n$ and the number of vertices in $E_k$ is at most $n-1$, we can make either color of $p$ or color $q$ be free and this color to be assigned to $u$. if such $E_k$ is not available then the number of vertices colored in $E_i \cup E_j < 1.25n$ .
\end{proof}

\begin{corollary}
	Let $H$ be a linear uniform hypergraph of size $n$. If $H$ has at least $\frac{n}{2}$ independent edges then $\chi(H) \leq 1.25n$.
\end{corollary}

%

\bibliographystyle{plain}
\bibliography{EFL1.25}

\begin{thebibliography}{1}

\bibitem{berge1990onvizing}
Claude Berge.
\newblock On two conjectures to generalize {V}izing's theorem.
\newblock {\em Matematiche (Catania)}, 45(1):15--23 (1991), 1990.
\newblock Graphs, designs and combinatorial geometries (Catania, 1989).

\bibitem{chang1988edge}
William~I. Chang and Eugene~L. Lawler.
\newblock Edge coloring of hypergraphs and a conjecture of erd{\"o}s, faber,
  lov{\'a}sz.
\newblock {\em Combinatorica}, 8(3):293--295, 1988.

\bibitem{faber2016linear}
V~Faber.
\newblock Linear hypergraph edge coloring generalizations of the efl
  conjecture.
\newblock {\em Bulletin of Math. Sci. and Applications}, 17:1--9, 2016.

\bibitem{faber2010uniformregular}
Vance Faber.
\newblock The {E}rd{\H o}s-{F}aber-{L}ov\'asz conjecture---the uniform regular
  case.
\newblock {\em J. Comb.}, 1(2):113--120, 2010.

\bibitem{kahn1992coloring}
Jeff Kahn.
\newblock Coloring nearly-disjoint hypergraphs with $n+o(n)$ colors.
\newblock {\em Journal of Combinatorial Theory, Series A}, 59(1):31--39, 1992.

\end{thebibliography}

\end{document}